\documentclass[12pt,leqno]{amsproc}
\usepackage{graphicx}
\usepackage{amscd}
\usepackage{amsmath}
\usepackage{amsfonts}
\usepackage{amssymb}
\theoremstyle{plain}

\newtheorem{cor}{Corollary}

\newtheorem{lem}{Lemma}

\newtheorem{thm}{Theorem}
\numberwithin{equation}{section}

\theoremstyle{definition}
\newtheorem{exam}{Example}
\newtheorem{rmk}{Remark}

\newcommand{\lra}{{\longrightarrow}}

\newcommand{\I}{\item}
\newcommand{\II}{\begin{enumerate}}
\newcommand{\III}{\end{enumerate}}

\begin{document}

\title[A complete Hyting algebra whose Scott space is non-sober]
{A complete Heyting algebra whose Scott space is non-sober}
\author[X. Xu]{Xu Xiaoquan}
\address{School of Mathematics and Statistics\\
Minnan Normal University\\Fujian, Zhangzhou 363000, P.R. China} \email{xiqxu2002@163.com}

\author[X. Xi]{Xi Xiaoyong}
\address{School of mathematics and Statistics\\
Jiangsu Normal University\\ Jiangsu, Xuzhou, P.R. China}
\email{littlebrook@jsnu.edu.cn}

\author[D. Zhao]{Zhao Dongsheng}
\address{Mathematics and Mathematics Education\\
National Institute of Education Singapore\\
Nanyang Technological University\\
1 Nanyang Walk\\
Singapore 637616} \email{dongsheng.zhao@nie.edu.sg}


\begin{abstract}
We prove that (1) for any complete lattice $L$, the set
$\mathcal{D}(L)$ of all nonempty saturated compact subsets of the
Scott space of $L$ is a complete Heyting algebra (with the reverse
inclusion order); and (2) if the Scott space  of a complete lattice $L$
is non-sober, then the Scott space of $\mathcal{D}(L)$ is non-sober.
Using these results and the Isbell's example of a non-sober complete
lattice, we deduce that there is a complete Heyting algebra  whose
Scott space is non-sober, thus give a positive answer to a problem
posed by Jung. We will also prove that a $T_0$ space  is
well-filtered iff its upper space (the set $\mathcal{D}(X)$ of all
nonempty saturated compact subsets of $X$ equipped with the upper
Vietoris topology) is well-filtered, which answers another open
problem.
\end{abstract}

\maketitle

Sobriety and well-filteredness are two of the most important
properties for non-Hausdorff topological spaces. The Scott space of every
domain (continuous directed complete posets) is sober. Johnstone
\cite{johnstone-81} gave the first example of a dcpo whose Scott
space is non-sober. Soon after that, Isbell \cite{isbell}
constructed a complete lattice whose Scott space is non-sober.
The general problem in this line is whether each object in
a classic class of lattices has a sober Scott space. The Isbell's non-sober
complete lattice is not distributive. Thus recently Jung asked
whether there is  a distributive complete lattice whose Scott
space is non-sober. In this paper we shall give a positive answer
to Jung's problem. The main structure we shall use is the poset
$\mathcal{D}(X)$ of all nonempty saturated compact subsets of a
topological space $X$ equipped with the reverse inclusion order.
We first show that for any complete lattice $L$, the poset
$\mathcal{D}(L)$ of all nonempty saturated compact subsets of the
Scott space of $L$ is a complete Heyting algebra. Then we prove that
for a certain type $T_0$ spaces $X$, if $X$ is non-sober, then the
Scott space of $\mathcal{D}(X)$ is non-sober. An immediate
conclusion is that for any complete lattice $L$, if the Scott space
of $L$ is non-sober, then the Scott space of the complete Heyting
algebra $\mathcal{D}(L)$ is non-sober. Taking $L$ to be the Isbell's
example, we obtain a complete Heyting algebra whose Scott space is
non-sober, thus answer Jung's problem.

Heckmann and Keimel \cite{Klause-Heckmann} proved that a space $X$
is sober if and only if the upper space $\mathcal{D}(X)$ of $X$ is
sober. In \cite{xi-zhao-MSCS-well-filtered}, Xi and Zhao proved that
a space $X$ is well-filtered iff its upper space $\mathcal{D}(X)$ is a d-space.
But, it is still not known, as pointed out in
\cite{xi-zhao-MSCS-well-filtered}, whether it is true that a space
is well-filtered if and only if the upper space of $X$ is
well-filtered. In the last part of this paper we will give a
positive answer to this problem.

\section{Preliminaries}
A \emph{complete Heyting algebra} is a complete lattice $L$  satisfying  the following infinite distributive law:

\[ a{\wedge}{ \bigvee}\{a_i: i{\in} I\}={\bigvee}\{a{\wedge} a_i: i{\in} I\}\]

\noindent for any $a{\in} L$ and $\{a_i: i{\in} I\}{\subseteq} L$. Such a
complete lattice is also called a \emph{frame}. Apparently every
complete Heyting algebra is a distributive complete lattice.

An element $p$ of a meet-semilattice $S$ is a prime element if for any $a, b\in S$, $a\wedge b\le p$ implies $a\le p$ or $b\le p$. A frame (or complete Heyting algebra) $A$ is called {\sl spacial} if every element of $A$ can be expressed as a  meet of prime elements. It is well-known that a complete lattice $L$ is a spacial frame iff it is isomorphic to the lattice of all open subsets of some topological space (cf. \cite{johnstone-93}).

A subset $U$ of a poset $(P, \leq)$ is \emph{Scott open}  if i) $U$
is an upper set (that is, $U={\uparrow}U=\{x{\in} P: y{\le} x \mbox{ for some }
y{\in} U\}$), and ii) for any directed subset $D{\subseteq} P$,
$\bigvee\!\!D{\in} U$ implies $D\cap\!U{\not=}\emptyset$ whenever
$\bigvee\!\! D$ exists. All Scott open sets of a poset $P$ form a
topology on $P$, denoted by $\sigma(P)$ and called the \emph{Scott
topology} on $P$. The space $(P, \sigma(P))$ is denoted by
$\Sigma\!P$, called the \emph{Scott space} of $P$.

A poset is called a \emph{directed complete poset} (\emph{dcpo}, for
short) if every directed subset of the poset has a supremum. For
more about Scott topology and dcpos, see
\cite{Gierz-Cam}\cite{Jean-2013}.

 A subset $A$ of a topological space is called \emph{saturated} if $A$ equals the intersection of all open sets containing it.  A $T_0$  space $X$ is  \emph{well-filtered} if for any open set $U$ and filtered family  $\mathcal{F}$ of
 saturated compact subsets of $X$, $\bigcap\mathcal{F}{\subseteq} U$ implies $F{\subseteq} U$ for some $F{\in}\mathcal{F}$.

A subset $A$ of a space $X$ is \emph{irreducible} if for any closed
subsets $F_1, F_2$ of $X$,  $F{\subseteq} F_1{\cup} F_2$ implies
$F{\subseteq} F_1$ or $F{\subseteq} F_2$. Obviously, the closure of every singleton set
is irreducible. A space $X$ is called \emph{sober} if every
irreducible closed set of $X$ is the closure of a unique singleton set. It
is well known that every sober space is well-filtered. Jonstone
\cite{johnstone-81} constructed  the first example of a dcpo whose Scott
space is non-sober. Isbell \cite{isbell} constructed a complete
lattice whose Scott space is non-sober. Kou \cite{Kou} gave the
first example of a dcpo whose Scott space is well-filtered but
non-sober.

The \emph{specialization order} $\leq_{\tau}$ on a $T_0$ space $(X,
\tau)$ is defined by $x{\leq_{\tau}} y$ iff $x{\in} cl(\{y\})$, where
$cl(\{y\})$ is the closure of set $\{y\}$. A space $(X, \tau)$ is called a \emph{d-space} (or \emph{monotone convergence space}) if $(X, \leq_{\tau})$ is a dcpo
 and $\tau {\subseteq} \sigma{((X, \leq_\tau))}$ (cf. \cite{Gierz-Cam}).

\section{The existence of a complete Heyting algebra whose Scott space is non-sober}

For any topological space $X$, following Heckmann and
Keimel \cite{Klause-Heckmann} we shall use $\mathcal{D}(X)$ to
denote the set of all nonempty compact saturated subsets of $X$. The
\emph{upper Vietoris topology} on $\mathcal{D}(X)$ is  the topology
that has $\{\Box\,U: U{\in}\mathcal{O}(X)\}$ as a base,
where $ \Box\,U=\{K{\in} \mathcal{D}(X): K{\subseteq} U\}$. The set
$\mathcal{D}(X)$ equipped with the upper Vietoris topology is called
the \emph{Smyth power space} or \emph{upper space} of $X$ (cf.
\cite{Klause-Heckmann}\cite{Schalk}).

The specialization order on the upper space $\mathcal{D}(X)$ is the reverse inclusion order $\supseteq$.
In what follows, the partial order on $\mathcal{D}(X)$ we will concern is just the reverse inclusion order.

For a poset $P$, we shall use $\mathcal{D}(P)$ to denote the poset
of all nonempty compact saturated subsets of the Scott space $(P,
\sigma(P))$.

A space $X$ is called \emph{coherent} if the intersection of any two
compact saturated subsets in $X$  is compact.

\begin{lem}\label{k(L)is a distributive lattice}
For any complete lattice $L$, $\mathcal{D}(L)$ is a complete Heyting algebra.
\end{lem}
\begin{proof}
The Scott space $\Sigma L$ of $L$ is well-filtered by Xi and
Lawson\cite{Xi-Lawson-2017}, and is coherent by Jia and Jung
\cite{jia-Jung-2016}. We now show that the poset $\mathcal{D}(L)$ is a complete Heyting algebra.

(i) From that $(L, \sigma(L))$ is well-filtered, it follows that
$\mathcal{D}(L)$ is closed under filtered intersections, thus it is
a dcpo, in which the infimum of a directed subsets $\mathcal{K}$ of
$\mathcal{D}(L)$ is the intersection of $\mathcal{K}$.

(ii) Also $\Sigma L$ is coherent and every member of  $\mathcal{D}(L)$
contains the top element $1_L$  of $L$, so the intersection $K_1{\cap} K_2$ of any two members $K_1, K_2$  of $\mathcal{D}(L)$ is again a member of $\mathcal{D}(L)$, which equals the join
$K_1{\vee} K_2$ of $K_1$ and $K_2$ in $\mathcal{D}(L)$. Also $\mathcal{D}(L)$  has $L$ as the least element, and
$\{1_L\}$ as the top element. It follows that  $\mathcal{D}(L)$ is both a dcpo and a  join semilattice, and has a least element. Therefore $\mathcal{D}(L)$  is a complete lattice. In addition, for any $K_1, K_2\in\mathcal{D}(L)$, the meet $K_1{\wedge} K_2$ of $K_1, K_2$ in $\mathcal{D}(L)$ clearly equals $K_1{\cup} K_2$.

Now for any subfamily $\{K_i: i{\in} I\}{\subseteq} \mathcal{D}(L)$, from (i) and (ii) we can deduce that
$$\bigvee\{K_i: i{\in} I\}=\bigcap\{K_i: i{\in} I\}.$$

Then for any  $K{\in} \mathcal{D}(L)$ and $\{K_{i}:i{\in} I\}{\subseteq} \mathcal{D}(L)$, we have
\[
\begin{array}{rcl}
K{\wedge}{\bigvee_{i\in I}}K_{i}&=&K{\cup}(\bigcap_{i\in I} K_{i})\\
&=&\bigcap_{\i\in I}(K{\cup} K_{i})\\
&=&\bigvee_{i\in I}(K{\wedge} K_{i}).\end{array}\]

\noindent Hence $\mathcal{D}(L)$ is a complete Heyting algebra.
\end{proof}

For any $T_0$ space $(X, \tau)$, let $\xi_X: X\rightarrow
\mathcal{D}(X)$ be the \emph{canonical mapping} given by
$$\xi_X(x)=\uparrow\!x=\{y{\in} X: x{\leq_{\tau}} y\}.$$
It is easy to see that $\xi_X: (X, \leq_{\tau})\lra (\mathcal{D}(D), \supseteq)$ is an order embedding.

To emphasize the codomain, we shall use $\xi_X^{\sigma}$ to denote the corresponding mapping $\xi_X^{\sigma}: (X, \tau)\lra (\mathcal{D}(X), \sigma(\mathcal{D}(X)))$, where
$\xi^{\sigma}_X(x)=\uparrow\!x$ for each $x\in X$.

\begin{thm}\label{main theorem}
Let $X$ be  a $T_0$ space such that

(i) the upper Vietoris topology on $\mathcal{D}(X)$ is contained in $\sigma(\mathcal{D}(X))$ (that is, the upper Vietoris topology is weaker than the Scott topology);

(ii) the mapping $\xi_X^{\sigma}: (X, \tau)\lra (\mathcal{D}(X), \sigma(\mathcal{D}(X)))$ is continuous; and

(iii) $(\mathcal{D}(X), \sigma(\mathcal{D}(X)))$ is sober.

\noindent Then $X$ is sober.
\end{thm}
\begin{proof}
Let $F$ be a nonempty closed irreducible subset of $X$. Then, as
$\xi_{X}^{\sigma}$ is continuous, $\xi_X^{\sigma}(F)$ is an irreducible subset of
$(\mathcal{D}(X), \sigma(\mathcal{D}(X)))$. Therefore, there exists
$K\in \mathcal{D}(X)$ such that
$$cl_{\sigma(\mathcal{D}(X))}(\xi_X^{\sigma}(F))=\downarrow_{\mathcal{D}(X)}K (=\{A{\in} \mathcal{D}(X): K{\subseteq} A\}),$$
where $cl_{\sigma(\mathcal{D}(X))}$ is the closure operator in the Scott space $(\mathcal{D}(X), \sigma(\mathcal{D}(X)))$.

Claim 1. Every element of $K$ is an upper bound of $F$ in the poset $(X, \leq_{\tau})$.

In fact, let $k{\in} K$ and $x{\in} F$. Then $\uparrow\! k{\in} \mathcal{D}(X)$ and $\uparrow\!k {\subseteq} K$. In addition, ${\uparrow}x=\xi_X^{\sigma}(x){\in} {\downarrow_{\mathcal{D}(X)}}K$, so $\uparrow\!x{\supseteq} K$. Hence
$\uparrow\!x{\supseteq} K{\supseteq} \uparrow\!k$, which implies $x{\leq_{\tau}} k$.

Claim 2. $K$ has a least element.

If, on the contrary, for each $k\in K$, there is $s(k){\in} K$ such that $k{\not\leq_{\tau}} s(k)$. Then
$$ K{\subseteq}{ \bigcup}\{X {\setminus}\downarrow\! s(k): k{\in} K\}.$$

\noindent Thus $K{\in} \Box\bigcup\{X {\setminus }\downarrow\! s(k): k{\in}
K\}{\in} \sigma(\mathcal{D}(X))$ (by the assumption (i) in Theorem \ref{main theorem}), implying
$$K{\in} cl_{\sigma(\mathcal{D}(X))}(\xi_X^{\sigma}(F)){\cap} \Box\bigcup\{X \setminus \downarrow\! s(k): k{\in}
K\}.$$

\noindent Hence
$$cl_{\sigma(\mathcal{D}(X))}(\xi_X^{\sigma}(F))\cap \Box\bigcup\{X \setminus \downarrow\! s(k): k{\in}
K\}{\not=}\emptyset.$$

\noindent Therefore $\xi_X^{\sigma}(F)\cap
\Box\bigcup\{X \setminus \downarrow\! s(k): k{\in} K\}{\not=}\emptyset$.
Hence there exists $y{\in} F$ with $\uparrow\!\!y {\subseteq}\bigcup\{X \setminus
\downarrow\! s(k): k{\in} K\}$. It follows that $y{\not\leq_{\tau}}s(k)$ for
some $s(k){\in} K$. But this contradicts Claim 1 (every element of $K$
is an upper bound of $F$).

Therefore $K$ has a least element, say $s$. Then $K=\uparrow\! s$.

Claim 3. $F=cl_{X}(\{s\})$.
As $s$ is an upper bound of $F$ and $F$ is closed, we only need to confirm that $s{\in} F$.

Assume, on the contrary, that $s{\not\in} F$.
Then $K{\subseteq} X
{\setminus} F$, so
$$K{\in} cl_{\sigma(\mathcal{D}(X))}(\xi_X^{\sigma}(F))\cap \Box
(X{\setminus} F)$$
and $\Box(X{\setminus} F){\in}\sigma(\mathcal{D}(X)).$

Therefore $\xi_X^{\sigma}(F)\cap \Box (X {\setminus} F)\not=\emptyset$, which is
impossible.

Hence $s{\in} F$, thus $F=\downarrow\! s=cl_X(\{s\})$.

All these together show that $(X, \tau)$ is sober.
\end{proof}

\begin{rmk}\label{j_X continuous}  (1) For every  $T_0$ space $X$, the mapping $\xi_{X}: X\lra \mathcal {D}(X)$ (the upper space of $X$)
is a topological embedding (cf. \cite{Klause-Heckmann}).

(2) For any poset $P$, the mapping
$$\xi_{P}^{\sigma}: (P, \sigma(P))\lra (\mathcal{D}(P), \sigma(\mathcal{D}(P)))$$
 is continuous, i.e., it preserves all existing suprema.

(3) Every well-filtered space is a d-space.

(4) A $T_0$ space $X$ is well-filtered iff $\mathcal {D}(X)$ is a dcpo and the upper
Vietoris topology on $\mathcal{D}(X)$ is contained in
$\sigma(\mathcal{D}(D))$ (equivalently, the upper space $\mathcal{D}(X)$ is a d-space) (see \cite{Heckmann}\cite{xi-zhao-MSCS-well-filtered}). In general, the well-filteredness of $X$ is stronger than the condition that the upper
Vietoris topology on $\mathcal{D}(X)$ is contained in
$\sigma(\mathcal{D}(D))$.

\indent For example, consider the poset $\mathbb{N}$ of natural numbers with the usual order. Then every element in $\mathbb{N}$ is compact and so $\mathbb{N}$ is an algebraic poset. Hence $\Sigma \mathbb{N}$ ($\sigma (\mathbb{N})$ equals the \emph{Alexandroff topology} on $\mathbb{N}$) is locally compact and $\mathcal {D}(\mathbb{N})=\{{\uparrow} n : n\in \mathbb{N}\}$, which is isomorphic to $\mathbb{N}$. Therefore $\mathcal {D}(\mathbb{N})$ is not a dcpo. Now we have that the upper
Vietoris topology on $\mathcal{D}(\mathbb{N})$ equals the Scott topology
$\sigma(\mathcal{D}(\mathbb{N}))$ (and also equals the Alexandroff topology on $\mathcal{D}(\mathbb{N})$). But $\Sigma \mathbb{N}$ is not well-filtered.
\end{rmk}

\begin{exam}\label{well-filtered not sober}
Let $X$ be any non-countable set and $\tau$ be the co-countable topology on $X$.
Then $(X, \tau)$ is a $T_1$ space. Clearly, the nonempty compact (saturated) subsets of $(X, \tau)$ are exactly the nonempty finite subsets of $X$, that is,
$\mathcal{D}(X)=\{F : F$ is a nonempty finite subset of $X\}$. Every directed subset $\mathcal {E}$ of $
\mathcal{D}(X)$ has a largest element (which is the intersection of
$\mathcal{E}$), so $(X, \tau)$ is well-filtered but non-sober ($X$ is an
irreducible closed set but not the closure of any singleton set).
Clearly $\mathcal{D}(X)$ is a dcpo and every element in $\mathcal{D}(X)$ is
compact. Hence $\mathcal{D}(X)$ is an algebraic domain and $\sigma
(\mathcal{D}(X))$ (which equals the Alexandroff topology on $\sigma
(\mathcal{D}(X))$) is sober. For $(X, \tau)$, the conditions (i) and
(iii) in Theorem \ref{main theorem} are satisfied, but the
assumption (ii) does not hold. In this case, the sobriety of
$(\mathcal{D}(X), \sigma(\mathcal{D}(X)))$ does not imply the
sobriety of $(X, \tau)$.
\end{exam}

By Remark \ref{j_X continuous}\! (4), Theorem \ref{main theorem} can be restated as the following one.

\begin{thm}\label{the main theorem}
Let $(X, \tau)$ be  a $T_0$ space such that

(i) $X$ is well-filtered;

(ii) the mapping $\xi_X^{\sigma}: (X, \tau)\lra (\mathcal{D}(X), \sigma(\mathcal{D}(X)))$ is continuous; and

(iii) $(\mathcal{D}(X), \sigma(\mathcal{D}(X)))$ is sober.

\noindent Then $X$ is sober.
\end{thm}

By Theorem \ref{the main theorem} and Remark \ref{j_X continuous}, we
deduce the following.

\begin{cor}\label{poset version} For a dcpo $P$, if $(P, \sigma (P))$ is well-filtered and $(\mathcal{D}(P)$, $\sigma(\mathcal{D}(P)))$ is sober
(equivalently, the upper space $\mathcal{D}(P)$ is a d-space and $(\mathcal{D}(P), \sigma(\mathcal{D}(P)))$ is sober), then $(P, \sigma(P))$ is sober.
\end{cor}

By Xi and Lawson \cite{Xi-Lawson-2017}, for any complete lattice
$L$, $(L, \sigma(L))$ is well-filtered, hence the upper Vietoris topology on
$\mathcal{D}(L)$ is contained in $\sigma(L)$. Now applying Corollary
\ref{poset version}, we obtain the following.
\begin{thm}\label{sobriety of k(L)}
For any complete lattice $L$, if $(L, \sigma(L))$ is non-sober, then $(\mathcal{D}(L), \sigma(\mathcal{D}(L)))$ is non-sober.
\end{thm}

Now we are ready to answer Jung's problem mentioned in the introduction.
\begin{exam}\label{Jung OK}
In \cite{isbell}, Isbell constructed a complete lattice whose
Scott topology is non-sober, thus answered a question posed by
Johnstone
 in \cite{johnstone-81}. The Isbell's complete lattice is not
distributive. In one of his recent talk in Singapore, Achim Jung
asked whether there is a distributive complete lattice whose Scott
topology is non-sober. We now can give a positive answer to this
problem. Take $M$ be the complete lattice  constructed by Isbell and
let $L=\mathcal{D}(M)$. Then by Lemma \ref{k(L)is a distributive
lattice}, $L$ is a complete Heyting algebra. Since the Scott space
of $M$ is non-sober, by Theorem \ref{sobriety of k(L)}, the Scott
space of  $L$ is non-sober. Hence $L$ is a complete Heyting algebra
whose Scott space is non-sober.
\end{exam}

\begin{rmk}
For any $T_0$ space $(X, \tau)$, the poset $(\mathcal{D}(X), \supseteq)$ is a meet-semilattice, where the meet of $K_1, K_2\in \mathcal{D}(X)$ equals $K_1\cup K_2$.
Then clearly every principle filter $\uparrow\!x=\{y\in X: x\le_{\tau} y\}$ is a prime element  of $\mathcal{D}(X)$. In addition, for any $K\in \mathcal{D}(X)$,
$$K=\bigwedge\{\uparrow\!x: x\in K\},$$
 showing that every element of $\mathcal{D}(X)$ can be expressed as a meet of prime elements. Hence by Lemma \ref{k(L)is a distributive lattice}, $(\mathcal{D}(L), \supseteq)$ is actually a spacial frame for any complete lattice $L$ (see \cite{johnstone-93} for more about spatial frames).  Thus the non-sober complete Heyting algebra $L$ obtained in Example \ref{Jung OK} is also a spacial frame.
\end{rmk}

\section{Well-filteredness of upper spaces}
In this section, the symbol $\mathcal{D}(X)$ will denote the upper
space of  topological space $X$.

In \cite{Klause-Heckmann}, it is proved that a space $X$ is sober iff the upper space $\mathcal{D}(X)$ is sober.
In \cite{xi-zhao-MSCS-well-filtered}, it is proved that a $T_0$ space $X$ is well-filtered if and only if its upper space is a d-space.
As remarked in \cite{xi-zhao-MSCS-well-filtered}, it is still not known the answer to the following problem: Must the upper space $\mathcal{D}(X)$ be well-filtered if $X$ is well-filtered?

We now give a positive answer to the above problem.

\begin{lem}\label{rudin's lemma} \emph{(\cite{Klause-Heckmann})}
 Let $X$ be a topological space and $\mathcal{A}$ an
irreducible subset of the upper space $\mathcal{D}(X)$. Then every  closed set $C {\subseteq} X$  that
meets all members of $\mathcal{A}$  contains an irreducible closed subset $A$ that still meets all
members of $\mathcal{A}$.
\end{lem}

\begin{rmk}\label{Rudin's lemma}  The irreducible closed set $A$ in Lemma \ref{rudin's lemma} can
be take as a minimal one: if $A'$ is a closed set, $A'{\subseteq} A$
and meets all members of $\mathcal{A}$ , then $A'=A$ (see the proof of Lemma 3.1 in \cite{Klause-Heckmann}).
\end{rmk}

The following result can be verified straightforwardly (see e.g. [{\bf 9}, page 128] or the proof of Lemma 3.1 in \cite{jia-Jung-2016}).

\begin{lem}\label{union in DX}  If $\mathcal{K}{\subseteq} \mathcal{D}(X)$ is a nonempty compact
subset of $\mathcal{D}(X)$, then $\bigcup\mathcal{K}{\in}
\mathcal{D}(X)$.
\end{lem}

\begin{thm}
A topological space $X$ is well-filtered iff its upper space $\mathcal{D}(X)$
is well-filtered.
\end{thm}
\begin{proof}
We only need to show that if $X$ is well-filtered, then so is $\mathcal{D}(X)$.
Let $\{\mathcal{K}_t: t{\in} T\}$ be a filtered family of saturated
compact subsets of $\mathcal{D}(X)$, $\mathcal{U}=\bigcup\{\Box U_i:
i{\in} I\}$ an open set of $\mathcal{D}(X)$ such that
$$\bigcap\{\mathcal{K}_t: t{\in} T\}{\subseteq} \mathcal{U}.$$
Suppose that $\mathcal{K}_t{\not\subseteq} \mathcal{U}$ for all $t$,
that is, $\mathcal{K}_t{\cap} (\mathcal{D}(X){\setminus}
\mathcal{U}){\not=}\emptyset.$ Then as $\{\mathcal{K}_t: t{\in} T\}$ is
an irreducible subset of the space $\mathcal{D}(X)$, by Lemma
\ref{rudin's lemma} and Remark \ref{Rudin's lemma}, there is a
minimal closed irreducible subset $\mathcal{C}{\subseteq}
\mathcal{D}(X){\setminus}\mathcal{U}$ that meets  every $\mathcal{K}_t
(t{\in} T)$.

Let $\mathcal{C}=\bigcap\{\diamond C_j: j{\in} J\}$, where each $C_j$
is a closed subset of $X$ and $\diamond C_j=\{F{\in} \mathcal{D}(X):
F{\cap} C_j{\not=}\emptyset\}$. For each $t{\in} T$, let
$K_t=\bigcup(\mathcal{K}_t{\cap} \mathcal{C})$.  As $\mathcal{K}_t{\cap}
\mathcal{C}$ is nonempty and compact in $\mathcal{D}(X)$, by Lemma
 \ref{union in DX} we have that $K_t{\in} \mathcal{D}(X)$. Also
$\{K_t: t{\in} T\}$ is a filtered family of member of
$\mathcal{D}(X)$, thus $K=\bigcap\{K_t: t{\in} T\}$ is a member of
$\mathcal{C}(X)$  because $X$ is well-filtered.

Claim 1. $K{\not\in} \mathcal{U}$. \\
\indent Assume, on the contrary, that $K{\in} \mathcal{U}$. Then
$K=\bigcap\{K_t: t{\in} T\}\subseteq  U_i$ for some $i{\in} I$. Then, as
$X$ is well-filtered, $K_t{\subseteq} U_i$ holds for some $t{\in} T$. Then
$\mathcal{K}_t{\cap} \mathcal{C}{\subseteq} \Box U_i{\subseteq}
\mathcal{U}$, contradicting $\mathcal{C}{\subseteq}
\mathcal{D}(X){\setminus}\mathcal{U}$.

Claim 2. $K{\in} \bigcap\{\uparrow_{\mathcal{D}(X)}(\mathcal{K}_t{\cap} \mathcal{C}): t{\in} T\}.$

Suppose, on the contrary,  that $K{\not\in}
\bigcap\{\uparrow_{\mathcal{D}(X)}(\mathcal{K}_t{\cap} \mathcal{C}):
t{\in} T\}.$ Then there is $t_0{\in} T$ such that $K{\not\in}
\uparrow_{\mathcal{D}(X)}(\mathcal{K}_{t_0}{\cap} \mathcal{C})$. Thus,
for any $G{\in} \mathcal{K}_{t_{0}}{\cap} \mathcal{C}$, there exists
$e(G)\in K{\setminus} G$. Then $G{\cap} \downarrow\!e(G)=\emptyset$ (Note that $G$
is a saturated compact set). Now for any $G{\in}
\mathcal{K}_{t_{0}}\cap \mathcal{C}$ and any $t{\in} T$, since
$e(G){\in} K$ (so $e(G){\in} K_t$), we have $e(G){\in} \bigcup(\mathcal{K}_t{\cap}
\mathcal{C})$. Thus there exists $H_t{\in} \mathcal{K}_t{\cap}
\mathcal{C}$ such that $e(G){\in} H_t$, implying
$$ H_t{\in} \mathcal{K}_t{\cap}\mathcal{C}{\cap} \diamond({\downarrow} e(G)).$$
It follows that
$$\mathcal{K}_t{\cap}\mathcal{C}{\cap} \diamond(\downarrow\! e(G)){\not=}\emptyset, \mbox{ for all } t{\in }T.$$

By the minimality of $\mathcal{C}$, we have $\mathcal{C}{\cap }\diamond(\downarrow\!\! e(G))=\mathcal{C}$, which implies $\mathcal{C}{\subseteq} \diamond(\downarrow\! e(G))$.

Therefore $\mathcal{C}{\subseteq} \bigcap\{\diamond(\downarrow\!\! e(G)): G{\in} \mathcal{K}_{t_0}{\cap} \mathcal{C}\}$. Note that for any
$G{\in} \mathcal{K}_{t_0}{\cap} \mathcal{C}, G{\not\in} \diamond(\downarrow\! e(G))$.  Hence
$$\emptyset \neq\mathcal{K}_{t_0}{\cap} \mathcal{C}=\mathcal{K}_{t_0}{\cap} \mathcal{C}\cap\bigcap\{\diamond(\downarrow\! e(G)): G{\in} \mathcal{K}_{t_0}{\cap} \mathcal{C}\}=\emptyset.$$
\noindent This contradiction confirms  Claim 2.

Now $K{\in} \bigcap\{\uparrow_{\mathcal{D}(X)}\!(\mathcal{K}_t{\cap} \mathcal{C}): t{\in} T\}{\subseteq} \bigcap\{\mathcal{K}_t: t{\in} T\}{\subseteq} \mathcal{U}$, which implies
$K{\in}\mathcal{U}$. But this contradicts Claim 1.

All these together  deduce that there must be some $t_0{\in} T$ such that $\mathcal{K}_{t_0}{\subseteq} \mathcal{U}$.

Hence  $\mathcal{D}(X)$ is well-filtered. The proof is completed.
\end{proof}
\vskip 0.5cm
The following result collects some of the equivalent conditions for a
space to be well-filtered.

\begin{thm}
For any $T_0$ space $X$, the following statements are equivalent:
\II
\I[(1)] $X$ is well-filtered.
\I[(2)] The upper space $\mathcal{D}(X)$ of $X$
is a d-space.
\I[(3)] The upper space $\mathcal{D}(X)$ of $X$ is well-filtered.
\III
\end{thm}

\noindent{\bf Acknowledgement} (1) The first author is sponsored by NSFC (11661057), the Ganpo
555 project for leading talent of Jiangxi Provence and the Natural
Science Foundation of Jiangxi Provence, China (20161BAB2061004).

(2) The second author is sponsored by NSFC (11361028, 61300153, 11671008, 11701500, 11626207) and NSF Project of Jiangsu Province, China (BK20170483).

(3) The third author is sponsored by NIE AcRF (RI 3/16 ZDS).

\end{document}